\documentclass[12 pt, reqno]{amsart}
\usepackage{mathrsfs}
\usepackage{graphicx}
\usepackage{amsmath}
\usepackage{amsfonts}
\usepackage{amssymb}
\usepackage{hyperref} 
\newtheorem{Exam.}{Exam.}

\newtheorem{definition}{Definition}

\newtheorem{theorem}{Theorem}
\newtheorem{remark}{Remark}
\newtheorem{lemma}[theorem]{Lemma}
\newtheorem{proposition}[theorem]{Proposition}
\newtheorem{corollary}[theorem]{Corollary}
\newtheorem{conj}[theorem]{Conjecture}

\def\L{\mathscr{L}}

\newcommand{\bea}{\begin{eqnarray}}
\newcommand{\eea}{\end{eqnarray}}
\usepackage{amsthm,amsmath,amssymb,url,cite,color}

\def\red{\textcolor{red}}

\numberwithin{equation}{section}

\renewcommand{\o}{\omega}
\renewcommand{\O}{\Omega}
\newcommand{\lp}{<_P}

\def\JSP{\mathcal{JSP}}
\def\LSP{\mathcal{LSP}}
\DeclareMathOperator{\JS}{JS}
\DeclareMathOperator{\js}{js}
\DeclareMathOperator{\LS}{LS}
\DeclareMathOperator{\ls}{ls}
\DeclareMathOperator{\des}{des}

\renewcommand{\L}{\mathscr{L}}

\def\N{\mathbb N}

\begin{document}

\title{Jacobi-Stirling polynomials and $P$-partitions}
\author{Ira M. Gessel}
\address[Ira M. Gessel]{Department of Mathematics, Brandeis University, Waltham, MA 02453-2728, USA}
\email{gessel@brandeis.edu}

\author{Zhicong Lin}
\address[Zhicong Lin]{Department of Mathematics and Statistics, Lanzhou University, China, and  Universit\'{e} de Lyon; Universit\'{e} Lyon 1; Institut Camille Jordan; UMR 5208 du CNRS; 43, boulevard du 11 novembre 1918, F-69622 Villeurbanne Cedex, France}
\email{lin@math.univ-lyon1.fr}

\author{Jiang Zeng}
\address[Jiang Zeng]{Universit\'{e} de Lyon; Universit\'{e} Lyon 1; Institut Camille Jordan; UMR 5208 du CNRS; 43, boulevard du 11 novembre 1918, F-69622 Villeurbanne Cedex, France}
\email{zeng@math.univ-lyon1.fr}

\thanks{Ira Gessel's research is partially supported by NSA grant
H98230-10-1-0196.}
\thanks{Jiang Zeng's research is partially supported by the research project ANR-08-BLAN-0243-03.}
\date{\today}

\begin{abstract} 
We  investigate the diagonal generating function of 
the Jacobi-Stirling numbers  of the second kind $ \JS(n+k,n;z)$  by
generalizing  the analogous  results  for the Stirling and  Legendre-Stirling numbers.  More precisely,
letting
$\JS(n+k,n;z)=p_{k,0}(n)+p_{k,1}(n)z+\dots+p_{k,k}(n)z^k$, 
 we  show that $(1-t)^{3k-i+1}\sum_{n\geq0}p_{k,i}(n)t^n$ is  a polynomial in $t$ with nonnegative integral coefficients and 
 provide  combinatorial interpretations of the coefficients by  using Stanley's theory of $P$-partitions.  

 \end{abstract}
\maketitle

\section{Introduction}
Let $\ell_{\alpha,\beta}[y](t)$ be the Jacobi differential operator:
$$\ell_{\alpha,\beta}[y](t)=\frac{1}{(1-t)^\alpha(1+t)^\beta}\left( -(1-t)^{\alpha+1}(1+t)^{\beta+1} y'(t)\right)' .$$
It is well known  that the Jacobi polynomial $y=P_{n}^{(\alpha,\beta)}(t)$ is an  eigenvector 
for the differential operator $\ell_{\alpha,\beta}$ corresponding to 
$n(n+\alpha+\beta+1)$, i.e., 
$$\ell_{\alpha,\beta}[y](t)=n(n+\alpha+\beta+1)y(t).$$
For each $n \in \N$, the Jacobi-Stirling numbers  $\JS(n,k;z)$  of the second kind appeared  originally 
as the coefficients in the 
expansion of the $n$-th composite power of $\ell_{\alpha,\beta}$ (see \cite{ev}):
$$(1-t)^\alpha(1+t)^\beta \ell_{\alpha,\beta}^n[y](t)=
\sum\limits_{k=0}^{n} (-1)^k \JS(n,k;z)\left( (1-t)^{\alpha+k}(1+t)^{\beta+k} y^{(k)}(t)\right)^{(k)},$$
where   $z=\alpha+\beta+1$, and 
can also be  defined as the connection coefficients in
\begin{equation}\label{eq:JS}
x^n=\sum_{k=0}^{n}\JS(n,k;z)\prod_{i=0}^{k-1}(x-i(z+i)).
\end{equation}
The Jacobi-Stirling numbers $\js(n,k;z)$ of the first kind are defined by
\begin{equation}\label{eq:js}
\prod_{i=0}^{n-1}(x-i(z+i))=\sum_{k=0}^{n}\js(n,k;z)x^k.
\end{equation}
When $z=1$, the Jacobi-Stirling numbers become
the \emph{Legendre-Stirling numbers} \cite{ev2} of
the first and second kinds:
\begin{align}\label{defLS}
 \ls(n,k)=\js(n,k;1), \quad \LS(n,k)=\JS(n,k;1).
\end{align}
Generalizing the work of Andrews and Littlejohn~\cite{al} on Legendre-Stirling numbers, Gelineau and Zeng~\cite{gz}
studied the combinatorial interpretations of the Jacobi-Stirling numbers and remarked on the connection with  
Stirling numbers and central factorial numbers. 
Further properties of the Jacobi-Stirling numbers have been given by  Andrews, Egge,  Gawronski, and Littlejohn \cite{aegl}.

The Stirling numbers of the second and  first kinds  $S(n,k)$  and $s(n,k)$
  are defined by
\begin{align}\label{eq:stirling}
x^n=\sum_{k=0}^nS(n,k) \prod_{i=0}^{k-1}(x-i),\qquad
\prod_{i=0}^{n-1}(x-i)=\sum_{k=0}^ns(n,k)x^k.
\end{align}
The lesser known central factorial numbers \cite[p.~213--217]{ri} 
$T(n,k)$ and $t(n,k)$ are defined   by
\begin{align}\label{eq:cf2}
x^n=\sum_{k=0}^nT(n,k)\,x\prod_{i=1}^{k-1}\left(x+\frac{k}{2}-i\right),
\end{align}
and
\begin{align}\label{eq:cf1}
x\prod_{i=1}^{n-1}\left(x+\frac{n}{2}-i\right)=\sum_{k=0}^nt(n,k)x^k.
\end{align}

Starting from the fact  that for fixed $k$,  the Stirling number $S(n+k,n)$ can be written as  a polynomial in $n$ of degree $2k$ and there exist  nonnegative integers $c_{k,j}$, $1\leq j\leq k$, such that 
 \begin{align}
 \sum_{n\geq0} S(n+k,n)t^n&=\frac{\sum_{j=1}^k c_{k,j}t^{j}}{(1-t)^{2k+1}},
 \end{align}
 Gessel and Stanley~\cite{gs}  gave a combinatorial interpretation for the $c_{k,j}$ in terms of the descents in \emph{Stirling permutations}.
Recently, Egge \cite{eg} has given an analogous result for the Legendre-Stirling numbers, and 
Gelineau~\cite{ge} has made a preliminary study of the analogous problem for Jacobi-Stirling numbers.
In this paper, we will prove some  analogous  results for the diagonal generating function of Jacobi-Stirling numbers.  
As noticed in \cite{gz}, the leading coefficient of the polynomial
$\JS(n,k;z)$ is $S(n,k)$ 
and the constant term of $\JS(n,k;z)$ is the central factorial number of the second  kind
 with even indices $T(2n,2k)$.
Similarly,  the leading coefficient of the polynomial
$\js(n,k;z)$ is  $s(n,k)$ and
the constant term of  $\js(n,k;z)$ is the  central factorial number of the first kind
 with even indices $t(2n,2k)$.

 \begin{definition}
The \emph{Jacobi-Stirling polynomial} of the second kind is defined by
\begin{align}\label{dp}
 f_k(n;z):=\JS(n+k,n;z).
\end{align}
The  coefficient $p_{k,i}(n)$ of $z^i$ in $ f_k(n;z)$ is called the \emph{Jacobi-Stirling coefficient} of the second kind for $0\leq i\leq k$. Thus
\begin{equation}\label{e-f1}
f_k(n;z)=p_{k,0}(n)+p_{k,1}(n)z+\dots+p_{k,k}(n)z^k.
\end{equation}
\end{definition}

The main goal of  this paper is to prove Theorems \ref{theorem1} and \ref{main1} below.

\begin{theorem}
\label{theorem1}\label{thm1}
For each integer  $k$ and $i$ such that $0\leq i\leq k$, 
there is a polynomial 
$A_{k,i}(t)=\sum_{j=1}^{2k-i} a_{k,i,j}t^j$ with 
positive integer coefficients   
such that 
\begin{equation}\label{eqa0}
\sum_{n\geq0}p_{k,i}(n)t^n=\frac{A_{k,i}(t)}{(1-t)^{3k-i+1}}.
\end{equation}
\end{theorem}
In order to give a combinatorial interpretation for 
$a_{k,i,j}$, we introduce  the multiset
$$
M_k:=\{1,1,\bar{1}, 2, 2, \bar{2}, \dots, k, k, \bar{k}\},
$$ 
where the elements are ordered by  
\begin{align}
\bar{1}<1<\bar{2}<2\ldots<\bar{k}<k.
\end{align}
Let $[\bar{k}]:=\{\bar{1}, \bar{2}, \dots, \bar{k}\}$. 
For any  subset $S\subseteq [\bar k]$, we set 
 $M_{k,S}=M_k\setminus S$.

\begin{definition}
A  permutation  $\pi$ of $M_{k,S}$ is a Jacobi-Stirling permutation if whenever $u< v < w$ and $\pi(u)=\pi(w)$, we have $\pi(v)>\pi(u)$. 
We denote by $\JSP_{k,S}$ the set of  Jacobi-Stirling permutations of $M_{k,S}$ and 
$$
\JSP_{k,i}=\bigcup_{S\subseteq[\bar k]\atop |S|=i}\JSP_{k,S}.
$$
\end{definition}
For example, the Jacobi-Stirling permutations of $\JSP_{2,1}$ are: 
\begin{align*}
&22\bar{2}11, \;\bar{2}2211, \;\bar{2}1221,\; \bar{2}1122, \; 221\bar{2}1, \;122\bar{2}1, \;1\bar{2}221, 1\bar{2}122,\;
2211\bar{2}, \;1221\bar{2},\\
 &1122\bar{2}, \;11\bar{2}22,\;
2211\bar{1}, \;1221\bar{1},\; 1122\bar{1}, \;11\bar{1}22,\;
22\bar{1}11, \;\bar{1}2211, \;\bar{1}1221,\; \bar{1}1122.
\end{align*}
Let $\pi=\pi_1\pi_2\ldots\pi_m$ be a word on a totally ordered alphabet. We say that $\pi$ has a descent at $l$, where $1\le l \le m-1$,  if  $\pi_l>\pi_{l+1}$.
Let $\des\,\pi$ be the number of descents of $\pi$. The following is our main interpretation for the coefficients $a_{k,i,j}$.

\begin{theorem}\label{main1}
 For $k\geq 1,\,0\leq i\leq k$, and $1\leq j\leq 2k-i$, the coefficient $a_{k,i,j}$ is 
the number of Jacobi-Stirling permutations in $\JSP_{k,i}$ with $j-1$ descents.
\end{theorem}
The rest of this paper is organized as follows.
 In Section~2, we investigate some elementary properties of the Jacobi-Stirling polynomials
 and prove Theorem~1. In Section~3
 we  apply Stanley's  $P$-partition theory 
to derive a first interpretation of the integers $a_{k,i,j}$ and then 
reformulate  it  in terms of
  descents of Jacobi-Stirling permutations in Section~4.  In Section~5,  
 we construct Legendre-Stirling posets in order to prove  a  similar  result for the Legendre-Stirling numbers, and then 
  to deduce 
 Egge's result  for Legendre-Stirling numbers~\cite{eg}   in terms of
  descents of Legendre-Stirling permutations. A second proof of 
  Egge's result  is given by making a link to our result  for Jacobi-Stirling permutations, namely Theorem~2.  We end this paper with a conjecture on the real-rootedness of the polynomials 
  $A_{k,i}(t)$.

\section{Jacobi-Stirling polynomials}
\begin{proposition} \label{prop:1}
For $0\leq i\leq k$, 
the  Jacobi-Stirling coefficient $p_{k,i}(n)$ is a polynomial in $n$ of degree $3k-i$. Moreover,  the  leading coefficient 
of $p_{k,i}(n)$ is 
\begin{equation}\label{la}
\frac{1}{3^{k-i}2^i \, i!\,(k-i)!}
\end{equation}
for all $0\leq i \leq k$.
\end{proposition}
\begin{proof}
We proceed  by induction on $k\geq 0$. 
For $k=0$, we have $p_{0,0}(n)=1$ since $f_0(n)=\JS(n,n;z)=1$.
Let $k\geq1$ and 
suppose  that   $p_{k-1,i}$ is a polynomial in $n$ of degree $3(k-1)-i$ for $0\leq i\leq k-1$. 
From  \eqref{eq:JS}  we deduce the recurrence relation:
\begin{align}
\left\{ \begin{array}{l} \JS(0,0;z)=1, \qquad \JS(n,k;z)=0,   \text{ if } k \not\in\{1,\ldots,n\}, \\
\label{eqlnk1} \JS(n,k;z)= \JS(n-1,k-1;z)+k(k+z)\,\JS(n-1,k;z),  \text{ for }   n,k \geq 1. \end{array} \right.
\end{align}
Substituting in  \eqref{dp}  yields 
\begin{equation}\label{eqf}
f_k(n;z)-f_k(n-1;z)=n(n+z)f_{k-1}(n;z).
\end{equation}
It follows from \eqref{e-f1} that  for $0\leq i\leq k$, 
\begin{equation}\label{eq:p}
p_{k,i}(n)-p_{k,i}(n-1)=n^2p_{k-1,i}(n)+np_{k-1,i-1}(n).
\end{equation}
Applying  the induction hypothesis, we see that $p_{k,i}(n)-p_{k,i}(n-1)$ is a polynomial in $n$ of degree at most
$$
\max(3(k-1)-i+2,3(k-1)-(i-1)+1)=3k-i-1.
$$  
Hence $p_{k,i}(n)$ is a polynomial in $n$ of degree at most $3k-i$. It remains to determine the  coefficient of $n^{3k-i}$, say $\beta_{k,i}$.
Extracting the coefficient of $n^{3k-i-1}$ in  \eqref{eq:p} we have
$$
\beta_{k,i}=\frac{1}{3k-i}(\beta_{k-1,i}+\beta_{k-1,i-1}).
$$
Now it is fairly  easy to see that \eqref{la} satisfies the above recurrence.
\end{proof}
 
\begin{proposition} \label{th:p}
For all $k\geq 1$ and $0\leq i\leq k$, we have
\begin{align}\label{eq:root}
 p_{k,i}(0)=p_{k,i}(-1)=p_{k,i}(-2)=\cdots=p_{k,i}(-k)=0.
 \end{align}
\end{proposition}
\begin{proof}
We proceed by induction on $k$.
By definition, we have  
$$
f_1(n;z)=\JS(n+1, n; z)=p_{1,0}(n)+p_{1,1}(n)z.
$$ 
As  noticed in \cite[Theorem~1]{gz}, the leading coefficient of the polynomial
$\JS(n,k;z)$   is $S(n,k)$ and
the constant term is $T(2n,2k)$.
We derive from \eqref{eq:stirling} and \eqref{eq:cf2} that 
\begin{align*}
  p_{1,1}(n)&=S(n+1,n)=n(n+1)/2,\\
   p_{1,0}(n)&=T(2n+2,2n)=n(n+1)(2n+1)/{6}.
\end{align*}
Hence \eqref{eq:root} is true for $k=1$.
Assume that \eqref{eq:root} is true for 
some  $k\geq 1$.  By \eqref{eq:p} we have 
$$
p_{k,i}(n)-p_{k,i}(n-1)=n^2p_{k-1,i}(n)+np_{k-1,i-1}(n).
$$
Since $\JS(0,k;z)=0$ if $k\geq 2$, we have $p_{k,i}(0)=0$.
The above equation and the induction hypothesis imply  successively that 
$$
p_{k,i}(-1)=0,\quad p_{k,i}(-2)=0,\;\ldots,\; p_{k,i}(-k+1)=0,\quad p_{k,i}(-k)=0.
$$
The proof is thus complete.
\end{proof}

\begin{lemma}\label{lem5}
For each integer  $k$ and $i$ such that $0\leq i\leq k$, there is a polynomial 
$A_{k,i}(t)=\sum_{j=1}^{2k-i} a_{k,i,j}t^j$ with integer  coefficients 
such that 
\begin{equation}\label{eqa}
\sum_{n\geq0}p_{k,i}(n)t^n=\frac{A_{k,i}(t)}{(1-t)^{3k-i+1}}.
\end{equation}
\end{lemma}
\begin{proof}
By Proposition~\ref{prop:1} and standard results concerning rational generating functions (cf. \cite[Corollary 4.3.1]{st}), 
for each integer  $k$ and $i$ such that $0\leq i\leq k$, there is a polynomial 
$A_{k,i}(t)=a_{k,i,0}+a_{k,i,1}t+\cdots +a_{k,i,3k-i}t^{3k-i}$ satisfying \eqref{eqa}.
Now, by \cite[Proposition 4.2.3]{st}, we have 
\begin{align}\label{eq:inverse}
\sum_{n\geq 1}p_{k,i}(-n)t^n=-\frac{A_{k,i}(1/t)}{(1-1/t)^{2k-i+1}}.
\end{align}
Applying \eqref{eq:root} we see that $a_{k,i,2k-i+1}=\cdots =a_{k,i,3k-i}=0$.
\end{proof}

 \begin{table}
\caption{The first values of $A_{k,i}(t)$}
\label{Lnkz}
\begin{tiny} \[ \begin{tabular}{c|cccccc}
$k\backslash i$ & $0$ & $1$ & $2$ & $3$
\\
\hline
$0$ & $1$ \\
\\
$1$ & ${t+t^2}$ & ${t}$  \\
\\
$2$ & ${t+14t^2+21t^3+4t^4}$ &${2t+12t^2+6t^3}$&${t+2t^2}$  \\
\\
$3$ & ${t+75t^2+603t^3+1065t^4+460t^5+36t^6}$ &  ${3t+114t^2+501t^3+436t^4+66t^5}$ &${3t+55t^2+116t^3+36t^4}$ &${t+8t^2+6t^3}$               \\
\end{tabular} \] \end{tiny} \end{table}

The first values of $A_{k,i}(t)$ are given in Table~1.
The following result gives a recurrence  for the coefficients $a_{k,i,j}$.
\begin{proposition}\label{prop6}
Let $a_{0,0,0}=1$.
For $k,i,j\geq 0$, we have 
the following recurrence for the integers $a_{k,i,j}$:
\bea\label{th:7}
\begin{gathered}
a_{k,i,j}=j^2a_{k-1,i,j}+[2(j-1)(3k-i-j-1)+(3k-i-2)]a_{k-1,i,j-1}\\
+(3k-i-j)^2a_{k-1,i,j-2}+ja_{k-1,i-1,j}+(3k-i-j)a_{k-1,i-1,j-1},
\end{gathered}
\eea
where $a_{k,i,j}=0$ if any of the indices $k,i,j$ is negative or if $j\notin \{1, \dots, 2k-i\}$.
\end{proposition}
\begin{proof}
For $0\leq i \leq k$, let
\begin{align}\label{fk}
F_{k,i}(t)=\sum_{n\geq0}p_{k,i}(n)t^n=\frac{A_{k,i}(t)}{(1-t)^{3k-i+1}}.
\end{align}
The recurrence relation \eqref{eq:p}  is equivalent to
\begin{align}\label{th:6}
F_{k,i}(t)=(1-t)^{-1}[t^2F''_{k-1,i}(t)+tF'_{k-1,i}(t)+tF'_{k-1,i-1}(t)]
\end{align}
with $F_{0,0}=(1-t)^{-1}$.  
Substituting \eqref{fk} into \eqref{th:6} we obtain
\begin{align*}
A_{k,i}(t)=&(1-t)^{3k-i}[t^2(A_{k-1,i}(t)(1-t)^{-(3k-i-2)})''\\
&+t(A_{k-1,i}(t)(1-t)^{-(3k-i-2)})'+t(A_{k-1,i-1}(t)(1-t)^{-(3k-i-1)})']\\
=&[t^2A''_{k-1,i}(t)(1-t)^2+2(3k-i-2)t^2A'_{k-1,i}(t)(1-t)\\
&+(3k-i-2)(3k-i-1)t^2A_{k-1,i}(t)]\\
&+[tA'_{k-1,i}(t)(1-t)^2+(3k-i-2)tA_{k-1,i}(t)(1-t)]\\
&+[tA'_{k-1,i-1}(t)(1-t)+(3k-i-1)tA_{k-1,i-1}(t)].
\end{align*}
Taking the coefficient of $t^j$ in both sides of the above equation, we have 
\begin{align*}
a_{k,i,j}=&j(j-1)a_{k-1,i,j}-2(j-1)(j-2)a_{k-1,i,j-1}+(j-2)(j-3)a_{k-1,i,j-2}\\
&+2(3k-i-2)(j-1)a_{k-1,i,j-1}-2(3k-i-2)(j-2)a_{k-1,i,j-2}\\
&+(3k-i-2)(3k-i-1)a_{k-1,i,j-2}+ja_{k-1,i,j}-2(j-1)a_{k-1,i,j-1}\\
&+(j-2)a_{k-1,i,j-2}+(3k-i-2)a_{k-1,i,j-1}-(3k-i-2)a_{k-1,i,j-2}\\
&+ja_{k-1,i-1,j}-(j-1)a_{k-1,i-1,j-1}+(3k-i-1)a_{k-1,i-1,j-1},
\end{align*}
which gives \eqref{th:7} after simplification.
\end{proof}

\begin{corollary}\label{cor7} For $k\geq 0$ and $0\leq i\leq k$,
 the coefficients $a_{k,i,j}$ are positive integers for $1\leq j\leq 2k-i$.
\end{corollary}
\begin{proof}
This  follows  from \eqref{th:7}  by induction on $k$. Clearly, this is true   for $k=0$ and $k=1$. 
Suppose that this is true for some $k\geq 1$.  As each term in the right-hand side of \eqref{th:7} is nonnegative, 
we only need to show that at least one term on the right-hand side of \eqref{th:7}  is strictly positive.
Indeed,  for $k\geq 2$, 
 the induction hypothesis  and \eqref{th:7}  imply that
\begin{itemize}
\item if $j=1$, then $a_{k,i,1}\geq a_{k-1,i-1,1}>0$; 
\item  if $2\leq j\leq 2k-i$, then
$a_{k,i,j}\geq (3k-i-j)a_{k-1,i-1,j-1}\geq ka_{k-1,i-1,j-1}>0$.
\end{itemize}
These two cases cover all possibilities.
\end{proof}

Theorem~\ref{thm1} follows then from Lemma~\ref{lem5}, Proposition~\ref{prop6} and Corollary~\ref{cor7}.\\

Now, define  the \emph{Jacobi-Stirling polynomial of the first kind} $g_k(n;z)$ by 
\begin{align}\label{eq:js1}
g_k(n;z)=\js(n,n-k;z).
\end{align} 

\begin{proposition}\label{re}
For $k\geq 1$, we have 
\begin{equation}\label{eq:relation}
g_k(n;z)=f_k(-n;-z).
\end{equation}
If we write 
$g_k(n;z)=q_{k,0}(n)+q_{k,1}(n)z+\dots+q_{k,k}(n)z^k$,
then 
\begin{align}\label{eq:first}
\sum_{n\geq 1}q_{k,i}(n)t^n=(-1)^k\frac{\sum_{j=1}^{2k-i}a_{k,i,3k-i+1-j}t^{j}}{(1-t)^{3k-i+1}}.
\end{align}
\end{proposition}
\begin{proof}
From \eqref{eq:js} we deduce
\begin{align}
\left\{ \begin{array}{l} \js(0,0;z)=1, \qquad \js(n,k;z)=0, \quad  \text{if } k \not\in\{1,\ldots,n\}, \\
\label{eqlnk2} \js(n,k;z)= \js(n-1,k-1;z)-(n-1)(n-1+z)\,\js(n-1,k;z),  \quad   n,k \geq 1. \end{array} \right.
\end{align}
It follows from  the above recurrence and \eqref{eq:js1} that
$$
g_k(n;z)-g_k(n-1;z)=-(n-1)(n-1+z)g_{k-1}(n-1;z).
$$
Comparing with \eqref{eqf} we get \eqref{eq:relation}, which
implies that  $q_{k,i}(n)=(-1)^ip_{k,i}(-n)$.  Finally  \eqref{eq:first}  follows from \eqref{eq:inverse}.
\end{proof}
\section{Jacobi-Stirling posets}
We first recall some basic facts about Stanley's theory of $P$-partitions 
(see \cite{st0} and \cite[\S 4.5]{st}).  
 Let $P$ be a poset, and let $\o$ be a labeling of $P$, i.e., an injection from $P$ to a totally ordered set (usually a set of integers). A \emph{$(P,\o)$-partition} (or $P$-partition if $\o$ is understood) is a function $f$ from $P$ to the positive integers satisfying 
\begin{enumerate}
\item if $x\lp y$ then $f(x)\le f(y)$
\item if $x\lp y$ and $\o(x) > \o(y)$ then $f(x) < f(y)$.
\end{enumerate}

A \emph{linear extension} of a poset $P$ is an extension of $P$ to a total order. We will identify a linear extension of $P$ labeled by $\o$ with the permutation obtained by taking the labels of $P$ in increasing order
with respect to the linear extension. For example, the linear extensions of the poset shown in Figure~\ref{f-P2} are $2\,1\,3$ and $2\,3\,1$. We write $\L(P)$ for the set of linear extensions of $P$ (which also depend on the labeling $\o$).

\begin{figure}[h]
\begin{center}
\setlength {\unitlength} {1mm}
\begin {picture} (101.5,21) \setlength {\unitlength} {1.2mm}
\thicklines
\put(42,4){\circle*{2}}\put(41,0){$2$}
\put(42,4){\line(-1,1){7}}\put(35,11){\circle*{2}}\put(34,13){$1$}
\put(42,4){\line(1,1){7}}\put(49,11){\circle*{2}}\put(48,13){$3$}
\end{picture}
\end{center}
\caption{A poset}
\label{f-P2}
\end {figure}

The \emph{order polynomial} $\O_P(n)$ of $P$ is the number of $(P,\o)$-partitions with parts in $[n]=\{1,2,\dots, n\}$. 
It is known that $\O_P(n)$ is a polynomial in $n$ whose degree is the number of elements of $P$. The following  is 
a fundamental result in the  $P$-partition theory \cite[Theorem 4.5.14]{st}:
\begin{equation}\label{ppa}
\sum_{n\geq1} \Omega_P(n) t^n= \frac{\sum_{\pi\in \L(P)}t^{\des\,\pi +1}}{(1-t)^{k+1}},
\end{equation}
where $k$ is the number of elements of $P$ and $\des\,\pi$ is computed according to the natural order of the integers.

For example, the two linear extensions of the poset shown in Figure \ref{f-P2} each have one descent, and the order polynomial for this poset is $2\binom{n+1}{3}$.
So equation \eqref{ppa} reads
\begin{equation*}
\sum_{n\geq1} 2\binom{n+1}{3}t^n= \frac{2t^2}{(1-t)^{4}}.
\end{equation*}

By \eqref{eqlnk1}  the Jacobi-Stirling numbers have the generating function
\begin{align}\label{eq:gfJS}
\sum_{n\geq0}\JS(n,k;z) t^n = \frac{t^k}{(1-(z+1)t)(1-2(z+2)t)\cdots (1-k(z+k)t)},
\end{align}
As $f_k(n;z)=\JS(n+k,n;z)$, switching $n$ and $k$ in the last equation yields
$$
\sum_{k\geq0}f_k(n;z) t^k = \frac{1}{(1-(z+1)t)(1-2(z+2)t)\cdots (1-n(z+n)t)}.
$$
Identifying the coefficients of  $t^k$ gives
\begin{equation} 
\label{e-J2}
f_k(n;z) = \sum_{1\le j_1\le j_2\le \cdots \le j_k\le n}
j_1 (z+j_1)\cdot j_2 (z+j_2) \cdots j_k (z+j_k).
\end{equation}

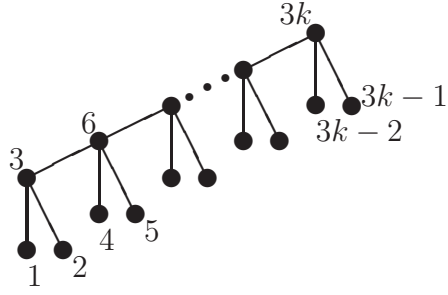
\begin{figure}
\begin {center}
\setlength {\unitlength} {1mm}
\begin {picture} (50,40) \setlength {\unitlength} {1.2mm}
\thicklines
\put(0,12){\circle*{2}}\put(0,12){\line(0,-1){8}}\put(0,4){\circle*{2}}\put(0,0){$1$}\put(-2,13){$3$}\put(0,12){\line(1,-2){4}}\put(4,4){\circle*{2}}\put(5,1){$2$}

\put(0,12){\line(2,1){8}}\put(8,16){\circle*{2}}\put(6,17){$6$}\put(8,16){\line(0,-1){8}}\put(8,8){\circle*{2}}\put(8,4){$4$}\put(8,16){\line(1,-2){4}}\put(12,8){\circle*{2}}\put(13,5){$5$}

\put(8,16){\line(2,1){8}}\put(16,20){\circle*{2}}\put(16,20){\line(0,-1){8}}\put(16,12){\circle*{2}}\put(18,21){\circle*{1}}\put(20,22){\circle*{1}}\put(22,23){\circle*{1}}\put(24,24){\circle*{2}}\put(16,20){\line(1,-2){4}}\put(20,12){\circle*{2}}

\put(24,24){\line(2,1){8}}\put(32,28){\circle*{2}}\put(28,29){$3k$}\put(24,24){\line(0,-1){8}}\put(24,16){\circle*{2}}\put(32,28){\line(0,-1){8}}\put(32,20){\circle*{2}}\put(32,16){$3k-2$}
\put(32,28){\line(1,-2){4}}\put(36,20){\circle*{2}}\put(37,20){$3k-1$}\put(24,24){\line(1,-2){4}}
\put(28,16){\circle*{2}}
\end{picture}
\end{center}
\caption{The labeled poset $R_k$.}
\label{f-C1}
\end {figure}

For any subset $S$ of $[k]$, we define $\gamma_{S,m}(j)$ by 
\begin{equation*}
\gamma_{S,m}(j) = 
\begin{cases}
j& \text{if $m\in S$},\\
j^2& \text{if $m\not\in S$},
\end{cases}
\end{equation*}
and define  $p_{k,S}(n) $ by
\begin{equation}\label{eqp}
p_{k,S}(n) = \sum_{1\le j_1\le j_2\le \cdots \le j_k\le n}
 \gamma_{S,1}(j_1)  \gamma_{S,2}(j_2)\cdots  \gamma_{S,k}(j_k).
\end{equation}


For example, if $k=2$ and $S=\{1\}$ then 
\begin{equation*}
p_{k,S}(n) = \sum_{1\le j_1\le j_2\le n} j_1 j_2^2=n(n+1)(n+2)(12n^2+9n-1)/120.
\end{equation*}


\begin{definition}
Let $R_k$ be the labeled poset in Figure~\ref{f-C1}. Let $S$ be a subset of $[k]$.  
 The poset $R_{k,S}$   obtained from $R_k$ by removing the points $3m-2$ for $m\in S$ is called a Jacobi-Stirling poset.
 \end{definition}

For example, the posets $R_{2,\{1\}}$ and $R_{2,\{2\}}$ are shown in Figure \ref{f-R1}.

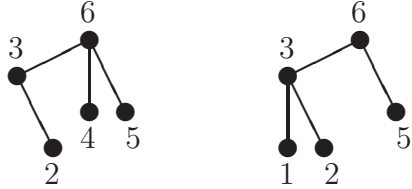
\begin{figure}[h!]
\begin {center}
\setlength {\unitlength} {1mm}
\begin {picture} (55,25) \setlength {\unitlength} {1.2mm}
\thicklines
\put(1,12){\circle*{2}}\put(1,12){\line(1,-2){4}}\put(5,4){\circle*{2}}
\put(0,14){$3$}\put(4,0){$2$}
\put(1,12){\line(2,1){8}}\put(9,16){\circle*{2}}\put(8,18){$6$}
\put(9,16){\line(0,-1){8}}\put(9,8){\circle*{2}}\put(8,4){$4$}
\put(9,16){\line(1,-2){4}}\put(13,8){\circle*{2}}\put(13,4){$5$}

\put(31,12){\circle*{2}}\put(31,12){\line(1,-2){4}}\put(35,4){\circle*{2}}
\put(31,12){\line(0,-1){8}}\put(31,4){\circle*{2}} \put(30,0){$1$}
\put(30,14){$3$}\put(35,0){$2$}
\put(31,12){\line(2,1){8}}\put(39,16){\circle*{2}}\put(38,18){$6$}
\put(39,16){\line(1,-2){4}}\put(43,8){\circle*{2}}\put(43,4){$5$}
\end{picture}
\end{center}
\caption{The labeled posets $R_{2,\{1\}}$ and $R_{2,\{2\}}$.}
\label{f-R1}
\end{figure}

\begin{lemma}\label{lea} 
For any subset $S\subseteq  [k]$, let  $A_{k,S}(t)$ be the descent polynomial of $\L(R_{k,S})$, i.e.,
the coefficient of $t^j$ in $A_{k,S}(t)$ is the number of linear extensions of $R_{k,S}$ with $j-1$ descents,
then 
\begin{equation}\label{AS}
\sum_{n\geq0} p_{k,S}(n) t^n = \frac{A_{k,S}(t)}{(1-t)^{3k-|S|+1}}.
\end{equation}
\end{lemma}
\begin{proof}
It is easy to see that $\O_{R_{k,S}}(n)=p_{k,S}(n)$ and the result follows from \eqref{ppa}.
\end{proof}

For $0\leq i\leq k$,   $R_{k,i}$ is defined as the set of ${k \choose i}$ posets
$$
R_{k,i}=\{\,R_{k,S}\ |\ S\subseteq[k]\text{ with cardinality $i$}\,\}.
$$ The posets in $R_{2,1}$ are shown in Figure \ref{f-R1}.
We define $\mathscr{L}(R_{k,i})$ to be the (disjoint)  union  of $\mathscr{L}(P)$, over all $P\in R_{k,i}$; i.e.,
\begin{align}
\mathscr{L}(R_{k,i})=\bigcup_{S\subseteq [k]\atop |S|=i}\mathscr{L}(R_{k,S}).
\end{align}

Now we are ready to give the first interpretation of the coefficients $a_{k,i,j}$ in the polynomial
$A_{k,i}(t)$  defined in  \eqref{eqa}.
\begin{theorem}\label{th:main}
We have 
\begin{align}\label{eq:AKI=AKS}
A_{k,i}(t)=\sum_{S\subseteq[k]\atop |S|=i}A_{k,S}(t).
\end{align}
In other words, the integer   $a_{k,i,j}$ is 
the number of elements   of $\mathscr{L}(R_{k,i})$  with $j-1$ descents.
\end{theorem}
\begin{proof} Extracting the coefficient of $z^i$ in both sides of \eqref{e-J2}, then
applying  \eqref{e-f1} and \eqref{eqp},   we obtain 
\begin{equation*}
p_{k,i}(n) = \sum_{S\subseteq[k] \atop |S|=i} p_{k,S}(n),
\end{equation*}
so that 
\begin{align*}
\sum_{n\geq0}p_{k,i}(n)t^n=\sum_{n\geq0} \sum_{S} p_{k,S}(n)t^n= \sum_S\sum_{n\geq0}p_{k,S}(n)t^n,
\end{align*}
where the summations on $S$ are over all subsets of $[k]$ with cardinality $i$.
 The result follows then by comparing \eqref{eqa} and \eqref{AS}.
\end{proof}

It is easy to compute $A_{k,S}(1)$ which is equal to $|\L(R_{k,S})|$ and is also  $(3k-i)!$ times the leading coefficient of $p_{k,S}(n)$. 

\begin{proposition}\label{prop:12}
 Let $S\subseteq [k],\,|S|=i$ and let $l_j(S)=|\{\,s\in S\ |\ s\leq j\,\}|$ for $1\leq j\leq k$. We have 
\begin{equation}\label{asl}
A_{k,S}(1)=\frac{(3k-i)!}{\prod_{j=1}^{k}(3j-l_j(S))}.
\end{equation}
\end{proposition}
\begin{proof}
We construct a permutation in $\L(R_{k,S})$ by reading  the elements of $R_{k,S}$  in increasing order of their labels and  inserting each one into the permutation already constructed from the earlier elements. Each element of $R_{k,S}$ will have two natural numbers associated to it:  the reading number
and the insertion-position number.
It is clear that the insertion-position number of $3j$ must be equal to its reading number, which is $3j - l_j(S)$, since it must be inserted to the right of all the previously inserted elements (those with labels less than $3j$). On the other hand, an element not divisible by 3 may be inserted anywhere, so its number of possible insertion positions is equal to 
its reading  number. So the number of possible linear extensions of $R_{k,S}$ is equal to the product of the reading numbers of all elements with labels not divisible by 3. Since the product of all the reading numbers is $(3j-i)!$, we obtain the result by dividing this number by the product of 
the reading numbers of the elements with labels $3, 6, \ldots, 3k$.
\end{proof}

From \eqref{asl}  we can derive  the formula for  $A_{k,i}(1)$, which is equivalent to Proposition~\ref{prop:1}.

\begin{proposition} 
\label{prop:A1}
We have
$$|\L(R_{k,i})|=A_{k,i}(1) = \frac{(3k-i)!}{3^{k-i}2^i \, i!\,(k-i)!}.$$
\end{proposition}
 \begin{proof}
 By Proposition \ref{prop:12} it is sufficient to prove the identity
 \begin{align}\label{eq:comid}
 \sum_{1\leq s_1<\cdots <s_i\leq k}\frac{(3k-i)!}{\prod_{j=1}^{k}(3j-l_j(S))}
   =\frac{(3k-i)!}{3^{k-i}2^i \, i!\,(k-i)!},
 \end{align}
where $S=\{s_1, \ldots, s_i\}$ and $l_j(S)=|\{\,s\in S\,: \,s\leq j\,\}|$.

 The identity is obvious  if  $S=\varnothing$, i.e., $i=0$. When $i=1$, it is easy to see that \eqref{eq:comid} is equivalent to the   $a=2/3$ case of 
 the indefinite summation
 \begin{align}\label{eq:identity}
 \sum_{s=0}^{k-1} \frac{(a)_s}{s!}=\frac{(a+1)_{k-1}}{(k-1)!},
 \end{align}
 where $(a)_n=a(a+1)\cdots (a+n-1)$ and $(a)_0=1$.
Since the left-hand side of \eqref{eq:comid} can be written as 
\begin{align}
\sum_{s_i=i}^k \frac{(3k-i)!}{\prod_{j=s_i}^k (3j-i)}\sum_{1\leq s_1<\cdots <s_{i-1}\leq s_i-1} \frac{1}{\prod_{j=1}^{s_i-1}(3j-l_j(S))},
\end{align}
we derive   \eqref{eq:comid}   from the induction hypothesis and \eqref{eq:identity}.
 \end{proof}
 
 \begin{remark}
 Alternatively,  we  may prove 
the formula for $A_{k,i}(1)$
as follows:
\begin{align*}
A_{k,i}(1)&=\sum_{S\subseteq[k]\atop |S|=i}A_{k,S}(1)\\
&=\sum_{S\subseteq[k]\atop |S|=i, k\in S} A_{k,S}(1)+\sum_{S\subseteq[k]\atop |S|=i, k\notin S} A_{k,S}(1)\\
&=(3k-i-1)A_{k-1,i-1}(1)+(3k-i-1)(3k-i-2)A_{k-1,i}(1),
\end{align*}
from which we easily deduce that $A_{k,i}(1)={(3k-i)!}/{3^{k-i}2^i\, i!\,(k-i)!}$.
\end{remark}

Since both of the above  proofs  of Proposition \ref{prop:A1} use mathematical induction, 
it is desirable to have  a more conceptual proof.
Here we give  such a proof based on the fact that Proposition \ref{prop:A1} is equivalent to 
\begin{align}\label{eq:prop13}
|\L(R_{k,i})|=2^{k-i} \cdot \frac{(3k-i)!}{2!^i \,i!\, 3!^{k-i}(k-i)!}.
\end{align}

\begin{proof}[A combinatorial proof of Proposition \ref{prop:A1}]
We show that 
$|\L(R_{k,i})|$  is equal to $2^{k-i}$ times the number of 
partitions of $[3k-i]$ with $k-i$ blocks of size 3 and $i$ blocks of size 2.

Let $S$ be an $i$-element subset of $[k]$ and let $\pi$ be an element $\L(R_{k,S})$, viewed as a bijection from  $[3k-i]$ to 
$R_{k,S}$.  
Let $\sigma=\pi^{-1}$. Then $\sigma$ is a natural labeling of $R_{k,S}$, i.e., an order-preserving bijection from the poset  $R_{k,S}$ to $[3k-i]$, and conversely, every natural labeling of  $R_{k,S}$ is the inverse of an element of $\L(R_{k,S})$.  

We will describe a map from the set of natural labelings of elements 
of $R_{k,i}$ to the set of partitions of $[3k-i]$ with $k-i$ blocks of size 3 and $i$ blocks of size 2, for which each such partition is the image of $2^{k-i}$ natural labelings. Given a natural labeling $\sigma$  of $R_{k,S}$,
the blocks of the corresponding partition are the sets $\{\sigma(3m-2), \sigma(3m-1), \sigma(3m)\}$ for $m\notin S$ and the sets $\{\sigma(3m-1), \sigma(3m)\}$ for $m\in S$. 
We note that since $\sigma$ is a natural labeling, $\sigma(3m)$ is always the largest element of its block and $\sigma(3)<\sigma(6)<\cdots<\sigma(3m)$. 

Now let $P$ be a partition of $[3k-i]$ with $k-i$ blocks of size 3 and $i$ blocks of size 2. We shall describe all natural labelings $\sigma$ of posets $R_{k,S}$ that correspond to $P$ under the map just defined.
First, we list the blocks of $P$ as $B_1$, $B_2$, \dots, $B_k$  in increasing order of their largest elements. Then $\sigma(3m)$ must be the largest element of $B_m$. If $B_m$ has two elements,  then the smaller element must be $\sigma(3m-1)$, and $m$ must be an element of $S$. If $B_m$ has three elements then $m\notin S$, and $\sigma(3m-2)$ and $\sigma(3m-1)$ are the two smaller elements of $B_m$, but in either order. Thus 
$S$ is uniquely determined by $P$, and there are exactly $2^{k-i}$ natural labelings of $R_{k,S}$ in the preimage of $P$. So $|\L(R_{k,i})|$ is $2^i$ times the number of partitions of $[3k-i]$ with $k-i$ blocks of size 3 and $i$ blocks of size 2, and is therefore equal to the right-hand side  of \eqref{eq:prop13}.
\end{proof}

\section{Two proofs of Theorem~\ref{main1}}
We shall give two proofs of Theorem~\ref{main1}. 
We first derive  Theorem~\ref{main1} from Theorem~\ref{th:main} by constructing a bijection from 
the linear extensions of Jacobi-Stirling posets to permutations.
The second proof consists of verifying that the cardinality  of Jacobi-Stirling permutations in $\JSP_{k,i}$ with $j-1$ descents satisfies the recurrence relation \eqref{th:7}.
 Given a word $w=w_1w_2\ldots w_m$ of $m$ letters,  we define the $j$th slot of $w$ by the pair $(w_j, w_{j+1})$ for $j=0, \ldots, m$. By convention 
$w_0=w_{m+1}=0$.  A slot $(w_j, w_{j+1})$  is called a descent (resp. non-descent) slot if $w_j>w_{j+1}$ (resp. $w_j\leq w_{j+1}$).

\subsection{ First  proof of Theorem~\ref{main1}}
For any subset $S=\{s_1, \ldots, s_i\}$ of $[k]$ we define  $\bar S=\{\bar s_1, \ldots, \bar s_i\}$, which is a  subset of $[\bar k]$.
Recall that $\JSP_{k,\bar S}$ is  the set of Jacobi-Stirling permutations of $M_{k,\bar S}$.
We construct a bijection  $\phi : {\L}(R_{k,S})\to {\JSP}_{k,\bar S}$ such that  $\des\, \phi (\pi)=\des\, \pi$ for any $\pi\in {\L}(R_{k,S})$.

If $k=1$,  then  $\L(R_{1,0})=\{123, \;213\}$ and $\L(R_{1,1})=\{23\}$.
We define $\phi$ by
$$
\phi (123)=\bar{1}11,\ \phi (213)=11\bar{1},\ \phi (23)=11.
$$
Suppose that $k\geq 2$ and $\phi : {\L}(R_{k-1,S})\to \JSP_{k-1,\bar S}$  is defined for any $S\subseteq [k-1]$. 
If  $\pi\in{\L}(R_{k,S})$ with $S\subseteq [k]$, we consider the following two cases: 
\begin{itemize}
\item[(i)] $k\notin S$,  denote by $\pi'$ the word obtained by   deleting $3k$ and  $3k-1$ from $\pi$, and $\pi''$ the word obtained by further deleting $3k-2$ from $\pi'$. As
 $\pi''\in{\L}(R_{k-1,S})$, 
by induction hypothesis,  the permutation $\phi(\pi'')\in \JSP_{k-1,\bar S}$ is well defined. Now,
\begin{itemize}
\item[a)]
if   $3k-2$ is in the $r$th descent  (or nondescent) slot of $\pi''$, then we insert $\bar{k}$ in the $r$th descent (or nondescent) slot of $\phi(\pi'')$ and obtain a word $\phi_1(\pi'')$;
\item[b)]
if  $3k-1$ is  in the $s$th descent (or nondescent)  slot of $\pi'$, we define $\phi(\pi)$ by inserting  $kk$ in the $s$th descent (or nondescent) slot of $\phi_1(\pi'')$.
\end{itemize}
\item[(ii)] $k\in S$,  denote by $\pi'$  the word obtained from $\pi$ by deleting $3k$ and $3k-1$.
As $\pi'\in{\L}(R_{k-1,i-1})$, the permutation  $\phi(\pi')\in \JSP_{k-1,\bar S}$ is well defined.
If 
 $3k-1$ is  in the $r$th descent (or nondescent) slot of $\pi'$,  we define $\phi(\pi)$ by inserting $kk$ in the $r$th descent (or nondescent) slot of $\phi(\pi')$.
\end{itemize}

Clearly this mapping is a bijection and preserves the number of descents.
For example, if $k=3$ and $S=\{2\}$,  then $\phi(2513\red{78}6\red{9})=112\red{\bar{3}}2\red{33}\bar{1}$. This can be seen by applying the mapping $\phi$ as follows:
\begin{align*}
&213\rightarrow2\red{5}13\red{6}\rightarrow 2513\red{7}6\rightarrow 2513\red{78}6\red{9},\\
&11\bar{1}\rightarrow11\red{22}\bar{1}\rightarrow 112\red{\bar{3}}2\bar{1}\rightarrow 112\red{\bar{3}}2\red{33}\bar{1}.
\end{align*}
Clearly we have $\des(2{\bf 5}137{\bf 8}69)=2$ and $\des(112{\bf \bar{3}}23{\bf 3}\bar{1})=2$.
\subsection{ Second proof of Theorem~\ref{main1}}
Let  $\JSP_{k,i,j}$ be the set of Jacobi-Stirling permutations in $\JSP_{k,i}$ with $j-1$ descents.
Let  $a'_{0,0,0}=1$ and  $a'_{k,i,j}$ be the cardinality  of $\JSP_{k,i,j}$ for $k,i,j\geq 0$.
By definition, $a'_{k,i,j}=0$ if any of the indices $k,i,j<0$ or  $j\notin \{1, \dots, 2k-i\}$.
We show that 
$a'_{k,i,j}$'s  satisfy  the same recurrence \eqref{th:7} and initial conditions as $a_{k,i,j}$'s.

Any  Jacobi-Stirling permutation  of $\JSP_{k,i,j}$  can be obtained from one of the following five cases:
\begin{itemize}
\item[(i)] Choose a Jacobi-Stirling permutation in $\JSP_{k-1,i,j}$, insert $\bar{k}$ and then $kk$ in 
 one of the descent slots (an extra descent at the end of the permutation). Clearly, there are $a'_{k-1,i,j}$ ways to choose the initial permutation, $j$ 
 ways to insert $\bar{k}$, and $j$ ways to insert $kk$. 
\item[(ii)] Choose a Jacobi-Stirling permutation of $\JSP_{k-1,i,j-1}$, 
\begin{itemize}
\item[1)] insert $\bar{k}$ in a descent slot  and then $kk$ in  a non-descent slot. In this case, there are $a'_{k-1,i,j-1}$ ways to choose the initial permutation, $j-1$ ways to insert $\bar{k}$, and $3k-i-j-1$ ways to insert $kk$.
\item[2)] insert $\bar{k}$ in  a non-descent slot  and then $kk$ in a descent slot. In this case, there are $a'_{k-1,i,j-1}$ ways to choose the initial permutation, $3k-i-j-1$ ways to insert $\bar{k}$, and $j$ ways to insert $kk$.
\end{itemize}
\item[(iii)] Choose a Jacobi-Stirling permutation in $\JSP_{k-1,i, j-2}$,  insert $\bar{k}$ and then $kk$ in  one of the non-descent slots. In this case, there are $a'_{k-1,i,j-2}$ ways to choose the initial permutation, $3k-i-j$ ways to insert $\bar{k}$, and $3k-i-j$ ways to insert $kk$. 
\item[(iv)] Choose a Jacobi-Stirling permutation in $\JSP_{k-1,i-1, j}$ and insert $kk$ in one of the descent slots. There are $a'_{k-1,i-1,j}$ ways to choose the initial permutation, and $j$ ways to insert $kk$.
\item[(v)] Choose a Jacobi-Stirling permutation in $\JSP_{k-1,i-1, j-1}$ and  insert $kk$ in one of the non-descent slots. 
There are $a'_{k-1,i-1,j-1}$ ways to choose the initial permutation, and $3k-i-j$ ways to insert $kk$.
\end{itemize}
Summarizing all the above five  cases, we obtain
\begin{align*}
a'_{k,i,j}&=j^2a'_{k-1,i,j}+[2(j-1)(3k-i-j-1)+(3k-i-2)]a'_{k-1,i,j-1} \\
&+(3k-i-j)^2a'_{k-1,i,j-2}+ja'_{k-1,i-1,j}+(3k-i-j)a'_{k-1,i-1,j-1}.
\end{align*}
Therefore, the numbers $a'_{k,i,j}$  satisfy  the same recurrence and initial conditions as the $a_{k,i,j}$, so they are equal.

\section{Legendre-Stirling posets}

Let $P_k$ be the poset shown
in Figure~\ref{f-P},  called the \emph{Legendre-Stirling poset}.  
\begin{figure}[h]
\begin {center}
\setlength {\unitlength} {1mm}
\begin {picture} (50,40) \setlength {\unitlength} {1.2mm}
\thicklines
\put(0,12){\circle*{2}}\put(0,12){\line(0,-1){8}}\put(0,4){\circle*{2}}\put(0,0){$1$}\put(-2,13){$2$}\put(0,12){\line(1,-2){4}}\put(4,4){\circle*{2}}\put(5,1){$3$}

\put(0,12){\line(2,1){8}}\put(8,16){\circle*{2}}\put(6,17){$5$}\put(8,16){\line(0,-1){8}}\put(8,8){\circle*{2}}\put(8,4){$4$}\put(8,16){\line(1,-2){4}}\put(12,8){\circle*{2}}\put(13,5){$6$}

\put(8,16){\line(2,1){8}}\put(16,20){\circle*{2}}\put(16,20){\line(0,-1){8}}\put(16,12){\circle*{2}}\put(18,21){\circle*{1}}\put(20,22){\circle*{1}}\put(22,23){\circle*{1}}\put(24,24){\circle*{2}}\put(16,20){\line(1,-2){4}}\put(20,12){\circle*{2}}

\put(24,24){\line(2,1){8}}\put(32,28){\circle*{2}}\put(22.5,29.5){$3k-1$}\put(24,24){\line(0,-1){8}}\put(24,16){\circle*{2}}\put(32,28){\line(0,-1){8}}\put(32,20){\circle*{2}}\put(32,16){$3k-2$}
\put(32,28){\line(1,-2){4}}\put(36,20){\circle*{2}}\put(37,20){$3k$}\put(24,24){\line(1,-2){4}}
\put(28,16){\circle*{2}}
\end{picture}
\end{center}
\caption{The Legendre-Stirling poset $P_k$.}
\label{f-P}
\end {figure}
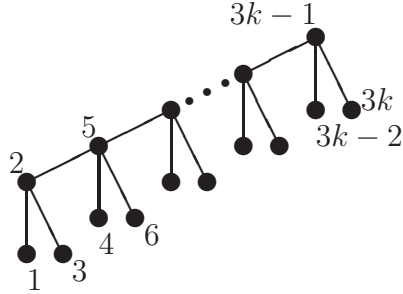
The order polynomial of $P_k$ is given by
\begin{align*}
\Omega_{P_k}(n)&=\sum_{2\leq f(2)\leq \dots \leq f(3k-1)\leq n}\prod_{i=1}^{k} f(3i-1)(f(3i-1)-1)\\
&=[x^k]\frac{1}{(1-2x)(1-6x)\dots (1-(n-1)nx)},
\end{align*}
which is equal to  
$\JS(n-1+k, n-1;1)$ by \eqref{eq:gfJS}, and by \eqref{defLS} this is equal to $\LS(n-1+k,n-1)$.
By   \eqref{ppa}, we obtain  
\begin{align}\label{LSP1} 
\sum_{n\geq0}\LS(n+k,n)t^n=\frac{\sum_{\pi\in \mathscr{L}(P_k)}t^{\des\,\pi}}{(1-t)^{3k+1}}.
\end{align}

In other words, we have the following theorem.
 \begin{theorem} \label{thm:lsposet}
 Let 
 $b_{k,j}$ be  the number of linear extensions of 
Legendre-Stirling posets $P_k$  with exactly $j$ 
descents.  Then
\begin{align}\label{LSP1} 
\sum_{n\geq0}\LS(n+k,n)t^n=\frac{\sum_{j=1}^{2k-1}b_{k,j}t^j}{(1-t)^{3k+1}}.
\end{align}

\end{theorem}

We now apply  the above theorem to deduce  a result of Egge~\cite[Theorem 4.6]{eg}.

\begin{definition}
A Legendre-Stirling permutation of $M_k$ is a Jacobi-Stirling permutation of 
$M_k$
with respect to the  order:  $\bar{1}=1<\bar{2}=2<\cdots<\bar{k}=k$.
\end{definition}

Here $\bar{1} = 1$ means that neither $1\bar{1}$ nor $\bar{1}1$ counts as a descent. Thus, the Legendre-Stirling permutation  $122\bar{2}1\bar{1}$   has one  descent  at  position  4, while 
as a Jacobi-Stirling permutation, it 
has three   descents, at  positions 3, 4 and  5.

 \begin{theorem}[Egge]
The coefficient  $b_{k,j}$ equals the number of 
Legendre-Stirling permutations of $M_k$ with exactly $j-1$ 
descents. 
\end{theorem}
\begin{proof}[First proof]
Let $\LSP_k$ be the set of Legendre-Stirling permutations of $M_k$. 
By Theorem~\ref{thm:lsposet}, it suffices to 
construct  a bijection  $\psi: \LSP_k\to \L(P_k)$ such that  $\des\, \psi (\pi)-1=\des\, \pi$ for any $\pi\in \LSP_{k}$.
If $k=1$, then $\LSP_1=\{11\bar 1, \; {\bar 1}11\}$ and  $\L(P_1)$=\{132,\;312\}. We define $\psi$ by 
$$
\psi(11\bar{1})=132,\quad \psi(\bar{1}11)=312.
$$
Clearly $\des\,132-1=\des\, 11\bar{1}=0$ and $\des\,312-1=\des\, \bar{1}11=0$.
Suppose that the bijection $\psi: \LSP_{k-1}\to \L(P_{k-1})$ is constructed for some $k\geq 2$.
Given  $\pi\in \LSP_{k}$, we  
denote by $\pi'$ the word obtained by deleting $\bar{k}$ from $\pi$, and by $\pi''$ the word obtained by further deleting $kk$ from $\pi'$. We put $3k-1$ at the end of $\psi(\pi'')$ and obtain a word $\psi_1(\pi'')$. In the following two steps, the slot after $3k-1$ is excluded, because we cannot insert $3k$ and $3k-2$ to the right of $3k-1$.
\begin{itemize}
\item[a)]
if   $\bar{k}$ is in the $r$th descent  (or nondescent) slot of $\pi''$, then we insert $3k$ in the $r$th descent (or nondescent) slot of $\psi_1(\pi'')$ and obtain a word $\psi_2(\pi'')$;
\item[b)]
if  $kk$ is  in the $s$th descent slot or in the non-descent slot before $\bar{k}$ (in the $j$th non-descent slot other than the non-descent slot before $\bar{k}$) of $\pi'$, we define $\psi(\pi)$ by inserting  $3k-2$ in the $s$th descent slot or in the non-descent slot before $3k$ (in the $j$th non-descent slot other than the non-descent slot before $3k$) of $\psi_2(\pi'')$.
\end{itemize}

For example, we can compute $\psi(\red{\bar{2}}12233\red{\bar{3}}1\red{\bar{1}})=\red{6}147\red{9}\red{3}258$ by the following procedure:
\begin{align*}
&11\red{\bar{1}}\rightarrow\red{\bar{2}}11\red{\bar{1}}\rightarrow\red{\bar{2}}12\red{2}1\red{\bar{1}}\rightarrow\red{\bar{2}}122\red{\bar{3}}1\red{\bar{1}}\rightarrow\red{\bar{2}}12233\red{\bar{3}}1\red{\bar{1}}\\
&1\red{3}2\rightarrow\red{6}1\red{3}25\rightarrow\red{6}1\red{4}\red{3}25\rightarrow\red{6}14\red{9}\red{3}258\rightarrow\red{6}147\red{9}\red{3}258.
\end{align*}
This construction can be easily reversed and the number of descents is preserved.
\end{proof}

\begin{proof}[Second proof]
By \eqref{dp}, \eqref{e-f1}, and \eqref{eqa0}, we have
$$
\sum_{n=0}^\infty \JS(n+k, n; z) t^n =
\sum_{i=0}^k z^i \frac{\sum_{j=1}^{2k-i}a_{k,i,j}t^j}{(1-t)^{3k-i+1}}.
$$
Setting $z=1$ and using \eqref{defLS} gives
$$
\sum_{n=0}^\infty \LS(n+k,n) t^n =
  \sum_{i=0}^k (1-t)^i \frac{\sum_{j=1}^{2k-i}a_{k,i,j}t^j}{(1-t)^{3k+1}}.
$$
Multiplying both sides by $(1-t)^{3k+1}$ and applying \eqref{LSP1}   gives 
$$
\sum_{j=1}^{2k-1}b_{k,j}t^j  =
  \sum_{i=0}^k (1-t)^i \sum_{j=1}^{2k-i}a_{k,i,j}t^j,
$$
so
\begin{equation}\label{JS-LS} 
\sum_{i=0}^{k}\sum_{l=0}^{i}(-1)^l{i \choose l}a_{k,i,j-l}=b_{k,j}.
\end{equation}
For any $S\subseteq [\bar k]$,
let $\JSP_{k,S,j}$ be the set of all Jacobi-Stirling permutations of $M_{k,S}$ with $j-1$ descents.
Let $B_{k,j}=\bigcup_{S\subseteq [\bar k]}\JSP_{k,S,j}$ be the set of Jacobi-Stirling permutations with $j-1$ descents.
We show that the left-hand side of \eqref{JS-LS} is the number $N_0$ of permutations in $B_{k,j}$ with no pattern $u\bar u$.

For any $T\subseteq [\bar k]$, let $B_{k,j}(T,\geq)$ be the set of permutations in $B_{k,j}$ containing all the patterns $u\bar u$ for $\bar u\in T$.
By the principle of inclusion-exclusion~\cite[Chapter~2]{st}, 
\begin{align}\label{PIE}
N_0=\sum_{T\subseteq [\bar k]} (-1)^{|T|} |B_{k,j}(T,\geq)|.
\end{align}
Now, for any subsets $T, S\subseteq [\bar k]$ such that $T\subseteq [\bar k]\setminus S$, define the mapping  
$$
\varphi: \JSP_{k,S,j}\cap B_{k,j}(T,\geq)\to  \JSP_{k,S\cup T,j-|T|}
$$ 
by deleting the $\bar u$ in every pattern $u\bar u$ of $\pi\in \JSP_{k,S,j}\cap B_{k,j}(T,\geq)$. Clearly, this is a bijection.
Hence, we can rewrite \eqref{PIE} as
\begin{align*}
N_0&=\sum_{T\subseteq [\bar k]} (-1)^{|T|}\sum_{S, T\subseteq [\bar k]\atop T\cap S=\emptyset} |\JSP_{k,S\cup T,j-|T|}|\\
&=\sum_{T\subseteq [\bar k]} (-1)^{|T|}\sum_{S\subseteq [\bar k]\atop T\subseteq S} |\JSP_{k,S,j-|T|}|.
\end{align*}
For any subset $S$ of $[\bar k]$ with $|S|=i$, and any $l$ with $0\leq l\leq i$, 
there are ${i\choose l}$  subsets $T$ of $S$ such that $|T|=l$,
and, by definition,
$$
\sum_{S\subseteq [\bar k]\atop |S|=i} |\JSP_{k,S,j-|T|}|=a_{k,i,j-l}.
$$
This proves that $N_0$ is equal to the left-hand side of \eqref{JS-LS}.

Let $\LSP_{k,j}$ be the set of all Legendre-Stirling permutations of $M_k$ with $j-1$ descents. 
It is easy to identify  a permutation  $\pi\in B_{k,j}$ with no pattern $u\bar u$ with a Legendre-Stirling permutation 
$\pi'\in \LSP_{k,j}$ by inserting each missing $\bar u$ just to the right of the second $u$. 
 This  completes the proof.
\end{proof}

Finally,  the numerical experiments suggest  the following conjecture,  which has been verified for $0\leq i\leq k\leq 9$.
\begin{conj}
For  $0\leq i\leq k$,  the polynomial  $A_{k,i}(t)$ has  only real roots.
\end{conj}

Note that by a classical result~\cite[p. 141]{co},  the above conjecture would imply that the sequence $a_{k,i,1}, \ldots, a_{k,i,2k-i}$ is unimodal.
Let $G_k$ be the multiset $\{1^{m_1}, 2^{m_2},\ldots, k^{m_k}\}$ with  $m_i\in\N$. A permutations $\pi$ of $G_k$  is a \emph{generalized Stirling permutation}  (see  \cite{br, ja}) if whenever $u< v < w$ and $\pi(u)=\pi(w)$, we have $\pi(v)>\pi(u)$.
 For any $S\subseteq [\bar{k}]$,  the set of generalized Stirling permutations of $M_k\setminus S$ is equal to $\JSP_{k,S}$. 
By Lemma~\ref{lea} and Theorem~\ref{main1},  the descent polynomial of $\JSP_{k,S}$ is $A_{k,S}(t)$. 
It follows from a result of  Brenti~\cite[Theorem 6.6.3]{br} that 
$A_{k,S}(t)$ has only real roots.  By  \eqref{eq:AKI=AKS}, 
this implies, in particular, that the above conjecture is true for $i=0$ and $i=k$.

 One can also use the methods of 
Haglund and Visontai~\cite{hv}  to show that $A_{k,S}(t)$ has only real roots, though it is not apparent how to use these methods to show that 
$A_{k,i}(t)$ has  only real roots.


\end{document}